\documentclass[A4paper,12pt]{amsproc}
\usepackage{graphicx}

\newtheorem{thm}{Theorem}
\newtheorem*{thm*}{Theorem}

\newtheorem{lem}{Lemma}

\newcommand{\as}{\text{ as }}

\renewcommand{\Re}{\text{Re \;}} \renewcommand{\Im}{\text{Im \;}}

\renewcommand{\l}{\left}
\renewcommand{\r}{\right}

\newcommand{\tr}{\text{tr}}

\begin{document}
\title[Weyl-Titchmarsh formula]
{Weyl-Titchmarsh type formula for periodic Schr\"odinger operator with Wigner-von Neumann potential}

\author{Pavel Kurasov}
    \address{Department of Mathematics, Lund Institute of Technology, 221 00 Lund, Sweden,}
    \email{kurasov@maths.lth.se}

\author{Sergey Simonov}
    \address{Department of Mathematical Physics, Institute of Physics, St. Petersburg University,
    Ulianovskaia 1, St. Petergoff, St. Petersburg, Russia, 198904}
    \email{sergey\_simonov@mail.ru}

\date{}

\begin{abstract}
Schr\"odinger operator on the half-line with periodic background
potential perturbed by a certain potential of Wigner-von Neumann
type is considered. The asymptotics of generalized eigenvectors
for $\lambda\in\mathbb C_+$ and on the absolutely continuous
spectrum is established. Weyl-Titchmarsh type formula for this
operator is proven.
\end{abstract}

\maketitle

\section{Introduction}
Consider the one dimensional Schr\"odinger operator with the real
potential which can be represented as a sum of three terms: a
certain periodic function, Wigner-von Neumann potential and a
certain absolutely integrable function. More precisely, let $q$ be
a real periodic function with period $a$ such that $q\in L_1(0;a)$
and let $q_1\in L_1(\mathbb R_+)$. Then the Schr\"odinger operator
$ \mathcal L_\alpha $ is defined by the differential expression
\begin{equation}\label{eq L-alpha}
    \mathcal L_{\alpha}:=-\frac{d^2}{dx^2}
    +q(x)+\frac{c\sin(2\omega x+\delta)}{(x+1)^{\gamma}}+q_1(x),
\end{equation}
on the set of functions satisfying the boundary condition
\begin{equation}\label{eq boundary condition}
    \psi(0)\cos\alpha-\psi'(0)\sin\alpha=0,
\end{equation}
where $c,\omega,\delta,\in\mathbb R$, $ \alpha \in [0, \pi)$,
$\gamma\in\l(\frac12;1\r]$. As it was shown by the first author
and Naboko in \cite{Kurasov-Naboko-07}, the absolutely continuous
spectrum of this operator has multiplicity one and coincides as a
set with the spectrum of the corresponding periodic operator on
$\mathbb R$,
\begin{equation}\label{eq L-per}
    \mathcal L_{per}=-\frac{d^2}{dx^2}+q(x).
\end{equation}
Note that the spectrum of $ \mathcal L_{per} $ has multiplicity
two. Let $\psi_+(x,\lambda)$ and $\psi_-(x,\lambda)$ be Bloch
solutions for $\mathcal L_{per}$ and $\varphi_{\alpha}(x,\lambda)$
be the solution of Cauchy problem
\begin{equation*}
    \begin{array}{l}
    -\varphi_{\alpha}''(x,\lambda)
    +\l(q(x)+\frac{c\sin(2\omega x+\delta)}{(x+1)^{\gamma}}+q_1(x)\r)\varphi_{\alpha}(x,\lambda)
    =\lambda\varphi_{\alpha}(x,\lambda),
    \\
    \varphi_{\alpha}(0,\lambda)=\sin\alpha,
    \\
    \varphi'_{\alpha}(0,\lambda)=\cos\alpha.
    \end{array}
\end{equation*}
The main result of the present paper is the following theorem that
relates the spectral density $\rho'_{\alpha}$ of the operator
$\mathcal L_{\alpha}$ and the asymptotics of the solution
$\varphi_{\alpha}$. We call it the Weyl-Titchmarsh formula.

    \begin{thm}\label{thm main result in introduction}
    Let $\frac{2a\omega}{\pi}\notin\mathbb Z$ and $q_1\in L_1(\mathbb
    R_+)$, then
    for almost all $\lambda\in\sigma(\mathcal L_{per})$ there exists $A_{\alpha}(\lambda)$ such that
    \begin{equation}\label{eq asymptotics of phi-alpha}
        \varphi_{\alpha}(x,\lambda)=A_{\alpha}(\lambda)\psi_-(x,\lambda)+
        \overline{A_{\alpha}(\lambda)}\psi_+(x,\lambda)+o(1)\as x\rightarrow\infty
    \end{equation}
    and
    \begin{equation*}
        \rho'_{\alpha}(\lambda)=\frac1{2\pi|W(\psi_+(\lambda),\psi_-(\lambda))| \; |A_{\alpha}(\lambda)|^2}.
    \end{equation*}
    \end{thm}

Weyl-Titchmarsh formulas form an efficient  tool to study the
behavior of the spectral density. The absolutely continuous
spectrum of the operator $\mathcal L_{\alpha}$ contains infinitely
many critical (resonance) points (see \eqref{eq critical points})
where the type of the asymptotics of generalized eigenvectors
changes and is not given by a linear conbination of $\psi_+$ and
$\psi_-$ (as in \eqref{eq asymptotics of phi-alpha}). Precisely at
these points the embedded eigenvalues of $\mathcal L_{\alpha}$ may
occur. In the generic case no eigenvalue occurs, but it is natural
to suspect that the spectral density of $\mathcal L_{\alpha}$
vanishes at these points.

Vanishing of the spectral density divides the absolutely
continuous spectrum into independent parts and has a clear
physical meaning. This phenomenon is called pseudogap. In the
forthcoming paper we intend to study zeros of the spectral density
in more detail.

The study of Schr\"odinger operators with Wigner-von Neumann
potentials began from the classical paper
\cite{Wigner-von-Neumann-29} where it was observed for the first
time that the potential $ \displaystyle  \frac{c\sin(2\omega
x+\delta)}{x+1} $ may produce an eigenvalue inside the absolutely
continuous spectrum. Later on such operators attracted attention
of many authors
\cite{Albeverio-72},\cite{Matveev-73},\cite{Matveev-Skriganov-72},\cite{Buslaev-Matveev-70},\cite{Behncke-91-I},\cite{Behncke-91-II},\cite{Behncke-94},\cite{Kurasov-92},\cite{Kurasov-96},\cite{Hinton-Klaus-Shaw-91},\cite{Klaus-91}.
The phenomenon of this nature, an embedded eigenvalue ("bound
state in the continuum"), was even observed experimentally in
semiconductor geterostructures \cite{Capasso-et-al-92}.

Weyl-Titchmarsh formula for the spectral density in the case of
zero periodic background potential follows directly from the
results of \cite{Matveev-73}. This formula was proved once again
in \cite{Behncke-91-I} where the method of Harris-Lutz
transformations \cite{Harris-Lutz-75} was used. In the present
paper we also use a modification of this method. We would like to
mention that another one approach was suggested in
\cite{Brown-Eastham-McCormack-98}, but again in the case of zero
periodic background potential.

\section{Preliminaries}

The spectrum of $\mathcal L_{per}$ consists of infinitely many
intervals \cite[Theorem 2.3.1]{Eastham-73}
\begin{equation*}
    \sigma(\mathcal L_{per}):=\bigcup\limits_{n=0}^{\infty}([\lambda_{2n};\mu_{2n}]\cup[\mu_{2n+1};\lambda_{2n+1}]),
\end{equation*}
where
\begin{equation*}
    \lambda_0<\mu_0\le\mu_1<\lambda_1\le\lambda_2<\mu_2\le\mu_3<\lambda_2\le\lambda_4<...,
\end{equation*}
where $ \lambda_j $ and $ \mu_j $ are the eigenvalues of the
Schr\"odinger differential equation on the interval $ [0,a ] $
with periodic and antiperiodic boundary conditions.  Spectral
properties of $\mathcal L_{per}$ are related to the entire
function $D(\lambda)$ (discriminant) and the function $k(\lambda)$
(quasi-momentum)
\begin{equation*}
    k(\lambda):=-i\ln\l(\frac{\tr D(\lambda)+\sqrt{\tr^2D(\lambda)-4}}2\r).
\end{equation*}
We can choose the branch of $k(\lambda)$ so that (this follows
from the properties of $D(\lambda)$, see \cite[Theorem
2.3.1]{Eastham-73})
\begin{equation*}
    \begin{array}{c}
    k(\lambda_0)=0, k(\mu_0)=k(\mu_1)=\pi, k(\lambda_1)=k(\lambda_2)=2\pi,...,
    \\
    k(\lambda)\in\mathbb R,\text{ if }\lambda\in\sigma(\mathcal L_{per}),
    \\
    k(\lambda)\in\mathbb C_+,\text{ if }\lambda\in\mathbb C_+.
    \end{array}
\end{equation*}
The eigenfunction equation for $\mathcal L_{per}$,
\begin{equation*}
    -\psi''(x)+q(x)\psi(x)=\lambda\psi(x),
\end{equation*}
has two solutions (Bloch solutions) $\psi_+(x,\lambda)$ and
$\psi_-(x,\lambda)$ satisfying quasiperiodic conditions:
\begin{equation*}
    \begin{array}{l}
    \psi_+(x+a,\lambda)\equiv e^{ik(\lambda)}\psi_+(x,\lambda),
    \\
    \psi_-(x+a,\lambda)\equiv e^{-ik(\lambda)}\psi_-(x,\lambda).
    \end{array}
\end{equation*}
They are determined uniquely up to multiplication by coefficients depending on $\lambda$.
It is possible to choose these coefficients so that Bloch solutions have the following properties:
\begin{enumerate}
\item $\psi_+(x,\lambda),\psi_-(x,\lambda)$ for every $x\ge0$ and
their Wronskian $W(\psi_+(\lambda),\psi_-(\lambda))$ are analytic
functions of $\lambda$ in $\mathbb C_+$ and continuous up to
$\sigma(\mathcal L_{per})\backslash\{\lambda_n,\mu_n,n\ge0\}$.
\item For $\lambda\in\sigma(\mathcal
L_{per})\backslash\{\lambda_n,\mu_n,n\ge0\}$,
\begin{equation*}
    \psi_+(x,\lambda)\equiv\overline{\psi_-(x,\lambda)}.
\end{equation*}
\item The Wronskian does not have zeros and for
$\lambda\in\sigma(\mathcal
L_{per})\backslash\{\lambda_n,\mu_n,n\ge0\}$
\begin{equation*}
    W(\psi_+(\lambda),\psi_-(\lambda))\in i\mathbb R_+.
\end{equation*}
\end{enumerate}

Bloch solutions can also be written in the form
\begin{equation*}
    \begin{array}{l}
    \psi_+(x,\lambda)=e^{ik(\lambda)\frac xa}p_+(x,\lambda),
    \\
    \psi_-(x,\lambda)=e^{-ik(\lambda)\frac xa}p_-(x,\lambda),
    \end{array}
\end{equation*}
where the functions $p_+(x,\lambda)$ and $p_-(x,\lambda)$ have
period $a$ in the variable $x$ and the same properties as
$\psi_+(x,\lambda)$ and $\psi_-(x,\lambda)$ with respect to the
variable $\lambda$.

As we mentioned earlier, the operator $\mathcal L_{\alpha}$ was
studied in \cite{Kurasov-Naboko-07}, where the asymptotics of the
generalized eigenvectors was obtained. The authors showed that in
every zone of $\sigma(\mathcal L_{per})$ ($[\lambda_n;\mu_n]$ if
$n$ is even and $[\mu_n;\lambda_n]$ if $n$ is odd) there exist two
critical points $\lambda_n^+$ and $\lambda_n^-$ determined by the
equalities
\begin{equation}\label{eq critical points}
    \begin{array}{l}
    k(\lambda_n^+)=\pi\l(n+1-\l\{\frac{a\omega}{\pi}\r\}\r),
    \\
    k(\lambda_n^-)=\pi\l(n+\l\{\frac{a\omega}{\pi}\r\}\r).
    \end{array}
\end{equation}
They do not coincide with each other and with the ends of zones, if
\begin{equation}\label{eq condition on the frequency}
    \frac{2a\omega}{\pi}\notin\mathbb Z.
\end{equation}

\section{Reduction of the spectral equation to the discrete linear
system of Levinson form}

In this section we transform the eigenufnction equation for
$\mathcal L_{\alpha}$ to a linear $2\times2$ system with the
coefficient matrix being a sum of the diagonal and summable
matrices.

Consider the eigenfunction equation for $\mathcal L_{\alpha}$:
\begin{equation}\label{eq spectral equation 2 time}
        -\psi''(x)+\l(q(x)+\frac{c\sin(2\omega x+\delta)}{(x+1)^{\gamma}}+q_1(x)\r)\psi(x)
        =\lambda\psi(x).
\end{equation}
For every
\begin{equation*}
    \lambda\in\mathbb C_+\cup\l(\sigma(\mathcal L_{per})\backslash\{\lambda_n,\mu_n,n\ge0\}\r)
\end{equation*}
let us make the following substitution
\begin{equation}\label{eq u}
\begin{array}{c}
    \left(%
    \begin{array}{c}
    \psi(x) \\
    \psi'(x) \\
    \end{array}%
    \right)
    =
    \left(%
    \begin{array}{cc}
    \psi_-(x,\lambda) & \psi_+(x,\lambda) \\
    \psi_-'(x,\lambda) & \psi_+'(x,\lambda) \\
    \end{array}%
    \right)
    u(x),
    \\
    \text{or}
    \\
    u(x):=\frac1{W(\psi_+(\lambda),\psi_-(\lambda))}
    \left(%
    \begin{array}{cc}
    \psi_+'(x,\lambda) & -\psi_+(x,\lambda) \\
    -\psi_-'(x,\lambda) & \psi_-(x,\lambda) \\
    \end{array}%
    \right)
    \left(%
    \begin{array}{c}
    \psi(x) \\
    \psi'(x) \\
    \end{array}%
    \right).
    \end{array}
\end{equation}
 Writing \eqref{eq spectral equation 2 time} as
\begin{equation*}
    \left(%
    \begin{array}{c}
    \psi(x) \\
    \psi'(x) \\
    \end{array}%
    \right)'
    =\left(%
    \begin{array}{cc}
    0 & 1 \\
    q(x)+\frac{c\sin(2\omega x+\delta)}{(x+1)^{\gamma}}+q_1(x)-\lambda & 0 \\
    \end{array}%
    \right)
    \left(%
    \begin{array}{c}
    \psi(x) \\
    \psi'(x) \\
    \end{array}%
    \right)
\end{equation*}
and substituting \eqref{eq u} into it, we get:
\begin{equation}\label{eq system for u}
    u'(x)=\frac{\frac{c\sin(2\omega x+\delta)}{(x+1)^{\gamma}}+q_1(x)}{W(\psi_+(\lambda),\psi_-(\lambda))}
    \left(%
    \begin{array}{l}
    -\psi_+(x,\lambda)\psi_-(x,\lambda)\ \ \ \ \ -\psi_+^2(x,\lambda) \\
    \ \ \ \ \ \psi_-^2(x,\lambda)\ \ \ \ \ \ \ \psi_+(x,\lambda)\psi_-(x,\lambda) \\
    \end{array}%
    \right)
    u(x).
\end{equation}
Let us introduce another one vector valued function $ v $
\begin{equation}\label{eq v}
    \begin{array}{c}
    v(x):=
    \left(%
    \begin{array}{cc}
    e^{-ik(\lambda)\frac xa} & 0 \\
    0 & e^{ik(\lambda)\frac xa} \\
    \end{array}%
    \right)
    u(x),
    \\
    \text{ or }
    \\
    u(x)=
    \left(%
    \begin{array}{cc}
    e^{ik(\lambda)\frac xa} & 0 \\
    0 & e^{-ik(\lambda)\frac xa} \\
    \end{array}%
    \right)
    v(x)
    \end{array}
\end{equation}
and the matrix
\begin{equation}\label{eq R-(1)}
    R^{(1)}(x,\lambda):=\frac{q_1(x)}{W(\psi_+(\lambda),\psi_-(\lambda))}
    \left(%
    \begin{array}{l}
    -p_+(x,\lambda)p_-(x,\lambda)\ \ \ \ \ -p_+^2(x,\lambda) \\
    \ \ \ \ \ p_-^2(x,\lambda)\ \ \ \ \ \ \ p_+(x,\lambda)p_-(x,\lambda) \\
    \end{array}%
    \right).
\end{equation}
Then the system \eqref{eq system for u} is equivalent to
\begin{multline}\label{eq system for v}
    v'(x)=
    \l[
    \left(%
    \begin{array}{cc}
    -\frac{ik(\lambda)}a & 0 \\
    0 & \frac{ik(\lambda)}a \\
    \end{array}%
    \right)
    +\frac{c\sin(2\omega x+\delta)}{(x+1)^{\gamma}W(\psi_+(\lambda),\psi_-(\lambda))}\r.
    \\
    \l.\times
    \left(%
    \begin{array}{l}
    -p_+(x,\lambda)p_-(x,\lambda)\ \ \ \ \ -p_+^2(x,\lambda) \\
    \ \ \ \ \ p_-^2(x,\lambda)\ \ \ \ \ \ \ p_+(x,\lambda)p_-(x,\lambda) \\
    \end{array}%
    \right)
    +R^{(1)}(x,\lambda)\r]v(x).
\end{multline}
Let us search for a differentiable matrix-valued function $Q(x)$
such that $Q(x),Q'(x)=O\l(\frac1{x^{\gamma}}\r)$ as
$x\rightarrow\infty$ and such that the substitution
\begin{equation}\label{eq definition of widetilde-v}
    v(x)=e^{Q(x)}\widetilde v(x)
\end{equation}
leads to a system for the vector valued function $\widetilde v$ of
the form
\begin{multline*}
    \widetilde v'(x)=
    \l[
    \left(%
    \begin{array}{cc}
    -\frac{ik}a & 0 \\
    0 & \frac{ik}a \\
    \end{array}%
    \right)
    +\frac{c\sin(2\omega x+\delta)}{(x+1)^{\gamma}W(\psi_+,\psi_-)}\r.
    \\
    \l.\times
    \left(%
    \begin{array}{cc}
    -p_+(x)p_-(x) & 0 \\
    0 & p_+(x)p_-(x) \\
    \end{array}%
    \right)
    +R^{(2)}(x)\r]\widetilde v(x),
\end{multline*}
where the remainder $R^{(2)}(x)$ also belongs to $L_1(0,\infty)$.
Using that
\begin{equation*}
    \begin{array}{l}
    e^{\pm Q(x)}=I\pm Q(x)+O\l(\frac1{x^{2\gamma}}\r),
    \\
    \l(e^{\pm Q(x)}\r)'=\pm Q'(x)+O\l(\frac1{x^{2\gamma}}\r)
    \end{array}
\end{equation*}
as $x\rightarrow\infty$ we obtain
\begin{multline}\label{eq equation for cancellation with commutator}
    \widetilde v'(x)=
    \l[
    \left(%
    \begin{array}{cc}
    -\frac{ik}a & 0 \\
    0 & \frac{ik}a \\
    \end{array}%
    \right)
    +\frac{c\sin(2\omega x+\delta)}{(x+1)^{\gamma}W(\psi_+,\psi_-)}\r.
    \\
    \l.\times
    \left(%
    \begin{array}{cc}
    -p_+(x)p_-(x) & -p_+^2(x) \\
    p_-^2(x) & p_+(x)p_-(x) \\
    \end{array}%
    \right)
    -Q'(x)
    \r.
    \\
    \l.
    -\l[Q(x),
    \left(%
    \begin{array}{cc}
    -\frac{ik}a & 0 \\
    0 & \frac{ik}a \\
    \end{array}%
    \right)
    \r]
    +R^{(1)}(x)+O\l(\frac1{x^{2\gamma}}\r)\r]\widetilde v(x),
\end{multline}
where
\begin{equation*}
    \l[Q(x),
    \left(%
    \begin{array}{cc}
    -\frac{ik}a & 0 \\
    0 & \frac{ik}a \\
    \end{array}%
    \right)
    \r]
\end{equation*}
is the commutator of the two matrices. Our aim is to cancel the
anti-diagonal entries of
\begin{equation*}
    \left(%
    \begin{array}{cc}
    -p_+(x)p_-(x) & -p_+^2(x) \\
    p_-^2(x) & p_+(x)p_-(x) \\
    \end{array}%
    \right)
\end{equation*}
in \eqref{eq equation for cancellation with commutator} by
properly choosing $Q$. To this end $Q$ has to satisfy the
following equation:
\begin{equation}\label{eq equation on Q}
    Q'(x)+\l[Q(x),
    \left(%
    \begin{array}{cc}
    -\frac{ik}a & 0 \\
    0 & \frac{ik}a \\
    \end{array}%
    \right)
    \r]
    =\frac{c\sin(2\omega x+\delta)}{(x+1)^{\gamma}W(\psi_+,\psi_-)}
    \left(%
    \begin{array}{cc}
    0 & -p_+^2(x) \\
    p_-^2(x) & 0 \\
    \end{array}%
    \right).
\end{equation}
The latter is equivalent (after multiplication by $ \displaystyle
     \left(%
    \begin{array}{cc}
    e^{-ik\frac xa} & 0 \\
    0 & e^{ik\frac xa} \\
    \end{array}%
    \right)
$
from the right and by its inverse from the left) to
\begin{multline}\label{eq intermediate equation on Q}
    \l(
    \left(%
    \begin{array}{cc}
    e^{ik\frac xa} & 0 \\
    0 & e^{-ik\frac xa} \\
    \end{array}%
    \right)
    Q(x)
    \left(%
    \begin{array}{cc}
    e^{-ik\frac xa} & 0 \\
    0 & e^{ik\frac xa} \\
    \end{array}%
    \right)
    \r)'
    \\
    =\frac{c\sin(2\omega x+\delta)}{(x+1)^{\gamma}W(\psi_+,\psi_-)}
    \left(%
    \begin{array}{cc}
    0 & -p_+^2(x)e^{2ik\frac xa} \\
    p_-^2(x)e^{-2ik\frac xa} & 0 \\
    \end{array}%
    \right).
\end{multline}
For every
\begin{equation*}
    \mu\in\sigma(\mathcal L_{per})\backslash\{\lambda_n,\mu_n,\lambda_n^+,\lambda_n^-,n\ge0\}
\end{equation*}
and for the values of $\lambda$ from some neighbourhood of the
point $\mu$ (which we will specify later) let us take the
following solution of \eqref{eq intermediate equation on Q}:
\begin{multline*}
    \l(
    \left(%
    \begin{array}{cc}
    e^{ik\frac xa} & 0 \\
    0 & e^{-ik\frac xa} \\
    \end{array}%
    \right)
    Q(x,\lambda,\mu)
    \left(%
    \begin{array}{cc}
    e^{-ik\frac xa} & 0 \\
    0 & e^{ik\frac xa} \\
    \end{array}%
    \right)
    \r)
    =\frac c{W(\psi_+(\lambda),\psi_-(\lambda))}
    \\
    \times
    \left(%
    \begin{array}{l}
    \ \ \ \ \ \ \ \ \ \ \ \ \ \ \ \ \ \ \ \ \ \ \ \ \
    0
    \ \ \ \ \ \ \ \ \ \ \ \ \ \ \ \ \ \ \ \ \ \ \ \ \
    \int\limits_x^{\infty}\frac{\sin(2\omega t+\delta)p_+^2(t,\lambda)e^{2ik(\lambda)\frac ta}dt}{(t+1)^{\gamma}}
    \\
    \int\limits_0^x\frac{\sin(2\omega t+\delta)p_-^2(t,\lambda)e^{-2ik(\lambda)\frac ta}dt}{(t+1)^{\gamma}}
    -\int\limits_0^{\infty}\frac{\sin(2\omega t+\delta)p_-^2(t,\lambda)e^{-2ik(\mu)\frac ta}dt}{(t+1)^{\gamma}}
    \ \ \ \ \ \ \ \ \
    0
    \\
    \end{array}%
    \right)
\end{multline*}
(this is our choice of constants of integration that depend on $\mu$). This leads to
\begin{multline}\label{eq formula for Q}
    Q(x,\lambda,\mu):=\frac c{W(\psi_+(\lambda),\psi_-(\lambda))}
    \\
    \times
    \left(%
    \begin{array}{l}
    \ \ \ \ \ \ \ \ \ \ \ \ \ \ \ \ \ \ \ \
    0
    \ \ \ \ \ \ \ \ \ \ \ \ \ \ \ \ \ \ \ \
    e^{-2ik(\lambda)\frac xa}
    \int\limits_x^{\infty}\frac{\sin(2\omega t+\delta)p_+^2(t,\lambda)e^{2ik(\lambda)\frac ta}dt}{(t+1)^{\gamma}}
    \\
    e^{2ik(\lambda)\frac xa}
    \l(
    \int\limits_0^x\frac{\sin(2\omega t+\delta)p_-^2(t,\lambda)e^{-2ik(\lambda)\frac ta}dt}{(t+1)^{\gamma}}
    -\int\limits_0^{\infty}\frac{\sin(2\omega t+\delta)p_-^2(t,\lambda)e^{-2ik(\mu)\frac ta}dt}{(t+1)^{\gamma}}
    \r)
    \
    0
    \\
    \end{array}%
    \right).
\end{multline}
In particular, for $\lambda=\mu$,
\begin{multline}\label{eq formula for Q on the real line}
    Q(x,\mu,\mu)=\frac c{W(\psi_+(\mu),\psi_-(\mu))}
    \\
    \times
    \left(%
    \begin{array}{cc}
        0
        &
        e^{-2ik(\mu)\frac xa}
        \int\limits_x^{\infty}\frac{\sin(2\omega t+\delta)\psi_+^2(t,\mu)dt}{(t+1)^{\gamma}}
        \\
        -e^{2ik(\mu)\frac xa}
        \int\limits_x^{\infty}\frac{\sin(2\omega t+\delta)\psi_-^2(t,\mu)dt}{(t+1)^{\gamma}}
        &
        0
        \\
    \end{array}%
    \right).
\end{multline}
Formula \eqref{eq formula for Q on the real line} does not make
sense if $\mu\in\mathbb C_+$ due to the divergence of the integral
in the lower entry. But it has analytic continuation in its second
argument from the point $\mu$.

Let us denote
\begin{equation*}
    \varepsilon(\mu):=\frac12\min_{n\in\mathbb Z}
    \l\{\l|\frac{2k(\mu)}a+2\omega+\frac{2\pi n}a\r|,\l|\frac{2k(\mu)}a-2\omega+\frac{2\pi n}a\r|\r\}.
\end{equation*}
Consider some $\beta>0$ and the set
\begin{multline*}
    U(\beta,\mu):=\{\lambda\in\overline{\mathbb C_+}:2\varepsilon(\lambda)\ge\varepsilon(\mu),
    0\le\Im 2k(\lambda)/a\le1,
    \\
    |\Re  k(\lambda)-k(\mu)|\le\beta\Im  k(\lambda)\}.
\end{multline*}
The set $U(\beta,\mu)$ is compact and contains the point $\mu$.  For every $\beta_1<\beta$
it contains some neighbourhood of the vertex of the sector
\begin{equation*}
    |\Re \lambda-\mu|\le\frac{\beta_1}{k'(\mu)}\Im \lambda.
\end{equation*}
Note that $k'(\mu)$ is positive for
$\mu\in\sigma(\mathcal L_{per})\backslash\{\lambda_n,\mu_n,n\ge0\}$.

    \begin{thm}\label{thm about Q}
    Let $\beta>0$ and
    \begin{equation*}
    \mu\in\sigma(\mathcal L_{per})\backslash\{\lambda_n,\mu_n,\lambda_n^+,\lambda_n^-,n\ge0\}.
    \end{equation*}
    Then there exists $c_1(\beta,\mu,\gamma)$ such that for every $x\ge0$ and $\lambda\in U(\beta,\mu)$
    holds:
    \begin{equation*}
        \|Q(x,\lambda,\mu)\|,\|Q'(x,\lambda,\mu)\|<\frac{c_1(\beta,\mu,\gamma)}{(x+1)^{\gamma}}.
    \end{equation*}
    \end{thm}

\begin{proof}
Note first that
\begin{equation*}
    k(\lambda)\in\mathbb R\text{ for }\lambda\in\sigma(\mathcal L_{per})
\end{equation*}
and
\begin{equation*}
    k(\lambda)\in\mathbb C_+\text{ for }\lambda\in\mathbb C_+.
\end{equation*}

Let us denote the entries of $ Q(x, \lambda, \mu) $ as follows
\begin{equation*}
    Q(x,\lambda,\mu)=
    \left(%
    \begin{array}{cc}
    0 & Q_{12}(x,\lambda) \\
    Q_{21}(x,\lambda,\mu) & 0 \\
    \end{array}%
    \right).
\end{equation*}
Let us estimate first the entry $Q_{12}$. Let $f$ be a periodic
function with period $a$ such that $f\in L_1(0;a)$, its Fourier
coefficients will be denoted by $f_n$
\begin{equation*}
    f_n:=\frac1a\int_0^af(x)e^{-2\pi in\frac xa}dx.
\end{equation*}

    \begin{lem}\label{lem estimate of the tail}
    If
    \begin{equation*}
    \{f_n\}_{n=-\infty}^{+\infty}\in l^1(\mathbb Z)
    \end{equation*}
    and $\xi\in\overline{\mathbb C_+}$ is such that
    \begin{equation*}
    \frac{a\xi}{2\pi}\notin\mathbb Z,
    \end{equation*}
    then
    \begin{equation}\label{eq estimate from lemma}
    \l|e^{-i\xi x}\int_x^{\infty}\frac{e^{i\xi t}f(t)dt}{(t+1)^{\gamma}}\r|
    \le2\l(\sum_{n=-\infty}^{+\infty}\frac{|f_n|}{\l|\xi-\frac{2\pi n}a\r|}\r)\frac1{(x+1)^{\gamma}}
    \end{equation}
    ({\it i.e.} the expression on the left-hand side exists and is estimated by the right-hand side).
    \end{lem}

\begin{proof}
Consider $x_1>x$. Since the Fourier series converges absolutely,
we have
\begin{equation}\label{eq using Fourier in lemma}
    e^{-i\xi x}\int_x^{x_1}\frac{e^{i\xi t}}{(t+1)^{\gamma}}
    \l(\sum_{n=-\infty}^{+\infty}f_ne^{2\pi n\frac ta}\r)dt
    =
    \sum_{n=-\infty}^{+\infty}f_ne^{-i\xi x}
    \int_x^{x_1}\frac{e^{i\l(\xi+\frac{2\pi n}a\r)t}}{(t+1)^{\gamma}}dt.
\end{equation}
Integrating by parts and estimating the absolute value we get
\begin{multline*}
    \l|e^{-i\xi x}\int_x^{x_1}\frac{e^{i\l(\xi+\frac{2\pi n}a\r)t}dt}{(t+1)^{\gamma}}\r|
    \le\frac1{\l|\xi+\frac{2\pi n}a\r|}
    \\
    \times
    \l(\frac1{(x+1)^{\gamma}}+\frac1{(x_1+1)^{\gamma}}+\gamma\int_x^{x_1}\frac{dt}{(t+1)^{\gamma+1}}\r)
    =\frac2{\l|\xi+\frac{2\pi n}a\r|(x+1)^{\gamma}}.
\end{multline*}
Substituting into \eqref{eq using Fourier in lemma} yields:
\begin{equation*}
    \l|e^{-i\xi x}\int_x^{x_1}\frac{e^{i\xi t}f(t)dt}{(t+1)^{\gamma}}\r|
    \le
    2\l(\sum_{n=-\infty}^{+\infty}\frac{|f_n|}{\l|\xi-\frac{2\pi n}a\r|}\r)\frac1{(x+1)^{\gamma}}.
\end{equation*}
By Cauchy's criterion the integral
\begin{equation*}
    \int_x^{\infty}\frac{e^{i\xi t}f(t)dt}{(t+1)^{\gamma}}
\end{equation*}
exists and the desired estimate \eqref{eq estimate from lemma} follows.
\end{proof}

Formula \eqref{eq formula for Q} implies
\begin{multline}\label{eq equality for Q-12}
    Q_{12}(x,\lambda)
    \\
    =\frac{ce^{2i\omega x+i\delta}}{2iW(\psi_+(\lambda),\psi_-(\lambda))}
    e^{-i\l(\frac{2k(\lambda)}a+2\omega\r)x}\int_x^{\infty}
    \frac{p_+^2(t,\lambda)e^{i\l(\frac{2k(\lambda)}a+2\omega\r)t}dt}{(t+1)^{\gamma}}
    \\
    -\frac{ce^{-2i\omega x-i\delta}}{2iW(\psi_+(\lambda),\psi_-(\lambda))}
    e^{-i\l(\frac{2k(\lambda)}a-2\omega\r)x}\int_x^{\infty}
    \frac{p_+^2(t,\lambda)e^{i\l(\frac{2k(\lambda)}a-2\omega\r)t}dt}{(t+1)^{\gamma}}.
\end{multline}
Denote the Fourier coefficients of the function $p_+^2(\cdot,\lambda)$ by $b_n(\lambda)$,
\begin{equation*}
    b_n(\lambda):=\frac1a\int_0^ap_+^2(x,\lambda)e^{-2\pi in\frac xa}dx.
\end{equation*}
Then Lemma \ref{lem estimate of the tail} applied to to \eqref{eq
equality for Q-12} gives
\begin{multline}\label{eq estimate for Q-12}
    |Q_{12}(x,\lambda)|
    \le
    \frac{|c|}{|W(\psi_+(\lambda),\psi_-(\lambda))|}
    \frac1{(x+1)^{\gamma}}
    \\
    \times
    \sum_{n=-\infty}^{+\infty}|b_n(\lambda)|
    \l(\frac1{\l|\frac{2k(\lambda)}a+2\omega+\frac{2\pi n}a\r|}+
    \frac1{\l|\frac{2k(\lambda)}a-2\omega+\frac{2\pi n}a\r|}\r)
    \\
    \le
    \frac{2|c|}{\varepsilon(\mu)|W(\psi_+(\lambda),\psi_-(\lambda))|}
    \l(\sum_{n=-\infty}^{+\infty}|b_n(\lambda)|\r)
    \frac1{(x+1)^{\gamma}}.
\end{multline}

Let us estimate now the entry $Q_{21}$. Formula \eqref{eq formula
for Q} implies
\begin{multline*}
    Q_{21}(x,\lambda,\mu)=\frac{ce^{2ik(\lambda)\frac xa}}{W(\psi_+(\lambda),\psi_-(\lambda))}
    \l(
    \int\limits_0^x\frac{\sin(2\omega t+\delta)p_-^2(t,\lambda)e^{-2ik(\lambda)\frac ta}dt}{(t+1)^{\gamma}}
    \r.
    \\
    \l.
    -\int\limits_0^{\infty}\frac{\sin(2\omega t+\delta)p_-^2(t,\lambda)e^{-2ik(\mu)\frac ta}dt}{(t+1)^{\gamma}}
    \r)
    \\
    =\frac{ce^{2ik(\lambda)\frac xa}}{W(\psi_+(\lambda),\psi_-(\lambda))}
    \int\limits_0^x\frac{\sin(2\omega t+\delta)p_-^2(t,\lambda)
    \l(e^{-2ik(\lambda)\frac ta}-e^{-2ik(\mu)\frac ta}\r)dt}{(t+1)^{\gamma}}
    \\
    -\frac{ce^{2ik(\lambda)\frac xa}}{W(\psi_+(\lambda),\psi_-(\lambda))}
    \int\limits_x^{\infty}\frac{\sin(2\omega t+\delta)p_-^2(t,\lambda)e^{-2ik(\mu)\frac ta}dt}{(t+1)^{\gamma}}.
\end{multline*}
Denote
\begin{multline*}
    Q_{21}^{I}(x,\lambda,\mu)
    :=\frac{ce^{2ik(\lambda)\frac xa}}{W(\psi_+(\lambda),\psi_-(\lambda))}
    \\
    \times
    \int\limits_0^x\frac{\sin(2\omega t+\delta)p_-^2(t,\lambda)
    \l(e^{-2ik(\lambda)\frac ta}-e^{-2ik(\mu)\frac ta}\r)dt}{(t+1)^{\gamma}}
\end{multline*}
and
\begin{equation*}
    Q_{21}^{II}(x,\lambda,\mu)
    :=-\frac{ce^{2ik(\lambda)\frac xa}}{W(\psi_+(\lambda),\psi_-(\lambda))}
    \int\limits_x^{\infty}\frac{\sin(2\omega t+\delta)p_-^2(t,\lambda)e^{-2ik(\mu)\frac ta}dt}{(t+1)^{\gamma}},
\end{equation*}
so that
\begin{equation*}
    Q_{21}(x,\lambda,\mu)=Q_{21}^{I}(x,\lambda,\mu)+Q_{21}^{II}(x,\lambda,\mu).
\end{equation*}

The second term can be estimated in the same manner as
$Q_{12}(x,\lambda)$ using Lemma \ref{lem estimate of the tail}.
Denote by
\begin{equation*}
    \hat b_n(\lambda):=\frac1a\int_0^ap_-^2(x,\lambda)e^{-2\pi in\frac xa}dx
\end{equation*}
the Fourier coefficients of $p_-^2(\cdot,\lambda)$. Then
\begin{multline}\label{eq estimate for Q-21-II}
    |Q_{21}^{II}(x,\lambda,\mu)|
    \le
    \frac{|c|}{|W(\psi_+(\lambda),\psi_-(\lambda))|}\frac1{(x+1)^{\gamma}}
    \\
    \times
    \sum_{n=-\infty}^{+\infty}|\hat b_n(\lambda)|
    \l(\frac1{\l|\frac{2k(\lambda)}a-2\omega-\frac{2\pi n}a\r|}+
    \frac1{\l|\frac{2k(\lambda)}a+2\omega-\frac{2\pi n}a\r|}\r)
    \\
    \le
    \frac{2|c|}{\varepsilon(\mu)|W(\psi_+(\lambda),\psi_-(\lambda))|}
    \l(\sum_{n=-\infty}^{+\infty}|\hat b_n(\lambda)|\r)
    \frac1{(x+1)^{\gamma}}
\end{multline}
(using that $k(\mu)\in\mathbb R$ and
$k(\lambda)\in\overline{\mathbb C_+}$).

To estimate $Q_{21}^I$ we shall need the following lemma:

    \begin{lem}\label{lem estimate of the beginning}
    Let $\varepsilon,\beta>0$, then there exists $c_2(\varepsilon,\beta,\gamma)$ such that for
    every $\xi_1$ and $\xi_2$ such that
    \begin{equation*}
        \begin{array}{c}
        0\le\Im \xi_1\le1,\ |\xi_1|\ge\varepsilon,
        \\
        \xi_2\in\mathbb R,\ |\xi_2|\ge\varepsilon,
        \\
        |\Re  \xi_1-\xi_2|\le\beta\Im  \xi_1
        \end{array}
    \end{equation*}
    and for every $x\ge0$ holds:
    \begin{equation*}
        \l|e^{i\xi_1 x}\int_0^x\frac{\l(e^{-i\xi_1t}-e^{-i\xi_2t}\r)dt}{(t+1)^{\gamma}}\r|
        <\frac{c_2(\varepsilon,\beta,\gamma)}{(x+1)^{\gamma}}.
    \end{equation*}
    \end{lem}

\begin{proof}
Integrating by parts we get
\begin{multline*}
    e^{i\xi_1 x}\int_0^x\frac{\l(e^{-i\xi_1t}-e^{-i\xi_2t}\r)dt}{(t+1)^{\gamma}}
    =\frac{ie^{i\xi_1x}(\xi_1-\xi_2)}{i\xi_1\xi_2}
    +\frac{e^{i(\xi_1-\xi_2)x}}{i\xi_2(x+1)^{\gamma}}
    \\
    +\frac i{\xi_1(x+1)^{\gamma}}
    +\frac{\gamma e^{i\xi_1x}(\xi_1-\xi_2)}{i\xi_1\xi_2}
    \int_0^x\frac{e^{-i\xi_2t}dt}{(t+1)^{\gamma+1}}
    \\
    -\frac{\gamma e^{i\xi_1x}}{i\xi_1}
    \int_0^x\frac{\l(e^{-i\xi_1t}-e^{-i\xi_2t}\r)dt}{(t+1)^{\gamma+1}}.
\end{multline*}
Consider the new constant
\begin{equation*}
    c_3(\gamma):=\max_{x\ge0}x^{\gamma}e^{-x}.
\end{equation*}
For every $x\ge0$ and $\xi_1$ considered,
\begin{equation*}
    (\Im \xi_1)^{\gamma}e^{-\Im \xi_1x}\le\frac{c_3(\gamma)e^{\Im \xi_1}}{(x+1)^{\gamma}}
    \le\frac{ec_3(\gamma)}{(x+1)^{\gamma}}.
\end{equation*}
Using that
\begin{equation*}
    |\Re \xi_1-\xi_2|\le\beta\Im \xi_1,
\end{equation*}
which is equivalent to
\begin{equation*}
    |\xi_1-\xi_2|\le\sqrt{\beta^2+1} \Im \xi_1,
\end{equation*}
and the integral can be estimated as
\begin{multline*}
    \l|e^{i\xi_1 x}\int_0^x\frac{\l(e^{-i\xi_1t}-e^{-i\xi_2t}\r)dt}{(t+1)^{\gamma}}\r|
    \le\frac{2ec_3(\gamma)\sqrt{\beta^2+1}}{\varepsilon^2(x+1)^{\gamma}}+\frac2{\varepsilon(x+1)^{\gamma}}
    \\
    +\frac{\gamma}{\varepsilon}
    \l|e^{i\xi_1 x}\int_0^x\frac{\l(e^{-i\xi_1t}-e^{-i\xi_2t}\r)dt}{(t+1)^{\gamma+1}}\r|.
\end{multline*}
The last term can be split into three parts as follows:
\begin{multline*}
    \l|e^{i\xi_1 x}\int_0^x\frac{\l(e^{-i\xi_1t}-e^{-i\xi_2t}\r)dt}{(t+1)^{\gamma+1}}\r|
    \\
    \le e^{-\Im \xi_1x}
    \l[\int_0^{\frac1{\Im \xi_1}}+\int_{\frac1{\Im \xi_1}}^{\frac x2}+\int_{\frac x2}^x\r]
    \frac{\l|e^{-i\xi_1t}-e^{-i\xi_2t}\r|dt}{(t+1)^{\gamma+1}}.
\end{multline*}
Let us estimate these three integrals separately.
\begin{enumerate}
\item
\begin{equation*}
    e^{-\Im \xi_1x}\int_0^{\frac1{\Im \xi_1}}\frac{\l|e^{-i\xi_1t}-e^{-i\xi_2t}\r|dt}{(t+1)^{\gamma+1}}
    =
    e^{-\Im \xi_1x}\int_0^{\frac1{\Im \xi_1}}\frac{\l|e^{-i(\xi_1-\xi_2)t}-1\r|dt}{(t+1)^{\gamma+1}}.
\end{equation*}
Introduce the constant
\begin{equation*}
    c_4(\beta):=\max_{|x|\le\sqrt{\beta^2+1}}\frac{|e^x-1|}{|x|}.
\end{equation*}
Since for the first interval
\begin{equation*}
    |-i(\xi_1-\xi_2)t|\le\frac{|\xi_1-\xi_2|}{\Im \xi_1}\le\sqrt{\beta^2+1},
\end{equation*}
we have:
\begin{multline*}
    e^{-\Im \xi_1x}\int_0^{\frac1{\Im \xi_1}}\frac{\l|e^{-i\xi_1t}-e^{-i\xi_2t}\r|dt}{(t+1)^{\gamma+1}}
    \le
    e^{-\Im \xi_1x}\int_0^{\frac1{\Im \xi_1}}\frac{c_4(\beta)\sqrt{\beta^2+1}\Im \xi_1tdt}{(t+1)^{\gamma+1}}
    \\
    \le\frac{ec_3(\gamma)c_4(\beta)\sqrt{\beta^2+1}(\Im \xi_1)^{1-\gamma}}{(x+1)^{\gamma}}
    \int_0^{\frac1{\Im \xi_1}}\frac{dt}{(t+1)^{\gamma}}
    \\
    =\frac{ec_3(\gamma)c_4(\beta)\sqrt{\beta^2+1}((1+\Im \xi_1)^{1-\gamma}-(\Im \xi_1)^{1-\gamma})}
    {(x+1)^{\gamma}(1-\gamma)}
    \\
    \le
    \frac{2^{1-\gamma}ec_3(\gamma)c_4(\beta)\sqrt{\beta^2+1}}{(1-\gamma)(x+1)^{\gamma}}.
\end{multline*}
\item For $x\ge\frac2{\Im \xi_1}$ we have
\begin{multline*}
    e^{-\Im \xi_1x}\int_{\frac1{\Im \xi_1}}^{\frac x2}\frac{\l|e^{-i\xi_1t}-e^{-i\xi_2t}\r|dt}{(t+1)^{\gamma+1}}
    \\
    \le
    e^{-\Im \xi_1\frac x2}\int_{\frac1{\Im \xi_1}}^{\frac x2}
    \frac{\l(e^{\Im \xi_1\l(t-\frac x2\r)}+e^{-\Im \xi_1\frac x2}\r)dt}{(t+1)^{\gamma+1}}
    \\
    \le
    2e^{-\Im \xi_1\frac x2}\int_{\frac1{\Im \xi_1}}^{\infty}\frac{dt}{t^{\gamma+1}}
    \le
    \frac{2^{\gamma+1}ec_3(\gamma)}{\gamma(x+2)^{\gamma}}.
\end{multline*}
For $x<\frac2{\Im \xi_1}$ the integral is negative.
\item
\begin{equation*}
    e^{-\Im \xi_1x}\int_{\frac x2}^x\frac{\l|e^{-i\xi_1t}-e^{-i\xi_2t}\r|dt}{(t+1)^{\gamma+1}}
    \le
    \int_{\frac x2}^x\frac{2dt}{(t+1)^{\gamma+1}}<\frac{2^{\gamma+1}}{\gamma(x+2)^{\gamma}}.
\end{equation*}
\end{enumerate}

Combining these estimates we get
\begin{equation*}
    \l|e^{i\xi_1 x}\int_0^x\frac{\l(e^{-i\xi_1t}-e^{-i\xi_2t}\r)dt}{(t+1)^{\gamma+1}}\r|
    <\frac{c_2(\varepsilon,\beta,\gamma)}{(x+1)^{\gamma}}
\end{equation*}
with
\begin{multline*}
    c_2(\varepsilon,\beta,\gamma):=
    \frac{2ec_3(\gamma)\sqrt{\beta^2+1}}{\varepsilon^2}+\frac2{\varepsilon}
    \\
    +\frac{\gamma}{\varepsilon}
    \l(
    \frac{2^{1-\gamma}ec_3(\gamma)c_4(\beta)\sqrt{\beta^2+1}}{1-\gamma}
    +
    \frac{2^{\gamma+1}ec_3(\gamma)}{\gamma}
    +
    \frac{2^{\gamma+1}}{\gamma}
    \r).
\end{multline*}
This completes the proof of the lemma.
\end{proof}

Let us continue to estimate $Q_{21}^I$
\begin{multline*}
    Q_{21}^I(x,\lambda,\mu)
    =\sum_{n=-\infty}^{+\infty}
    \frac{c\hat b_n(\lambda)e^{i\delta+2i\omega+2\pi in\frac xa}}{2iW(\psi_+(\lambda),\psi_-(\lambda))}
    e^{i\l(\frac{2k(\lambda)}a-2\omega-\frac{2\pi n}a\r)x}
    \\
    \times
    \int_0^x\frac{
    \l(
    e^{-i\l(\frac{2k(\lambda)}a-2\omega-\frac{2\pi n}a\r)t}-
    e^{-i\l(\frac{2k(\mu)}a-2\omega-\frac{2\pi n}a\r)t}
    \r)
    dt}{(t+1)^{\gamma}}
    \\
    -\sum_{n=-\infty}^{+\infty}
    \frac{c\hat b_n(\lambda)e^{-i\delta-2i\omega+2\pi in\frac xa}}{2iW(\psi_+(\lambda),\psi_-(\lambda))}
    e^{i\l(\frac{2k(\lambda)}a+2\omega-\frac{2\pi n}a\r)x}
    \\
    \times
    \int_0^x\frac{
    \l(
    e^{-i\l(\frac{2k(\lambda)}a+2\omega-\frac{2\pi n}a\r)t}-
    e^{-i\l(\frac{2k(\mu)}a+2\omega-\frac{2\pi n}a\r)t}
    \r)
    dt}{(t+1)^{\gamma}}.
\end{multline*}
Applying Lemma \ref{lem estimate of the beginning} we get
\begin{equation}\label{eq estimate for Q-21-I}
    |Q_{21}^I(x,\lambda,\mu)|
    \le
    \frac{|c|c_2(\varepsilon(\mu),\beta,\gamma)}{|W(\psi_+(\lambda),\psi_-(\lambda))|}
    \l(\sum_{n=-\infty}^{+\infty}|\hat b_n(\lambda)|\r)
    \frac1{(x+1)^{\gamma}}.
\end{equation}
Therefore combining the estimates \eqref{eq estimate for Q-12},
\eqref{eq estimate for Q-21-I} and \eqref{eq estimate for Q-21-II}
the matrix $ Q $ can be estimated as follows
\begin{multline*}
    \|Q(x,\lambda,\mu)\|\le\frac1{(x+1)^{\gamma}}
    \frac{|c|}{|W(\psi_+(\lambda),\psi_-(\lambda))|}
    \\
    \times
    \sqrt{\frac4{\varepsilon^2(\mu)}\l(\sum_{n=-\infty}^{+\infty}|b_n(\lambda)|\r)^2
    +
    \l(\frac2{\varepsilon(\mu)}+c_2(\varepsilon(\mu),\beta,\gamma)\r)^2
    \l(\sum_{n=-\infty}^{+\infty}|\hat b_n(\lambda)|\r)^2}.
\end{multline*}
Let us estimate now the Fourier coefficients. For $n\neq0$ we
have:
\begin{equation*}
    b_n(\lambda)=\frac1a\int_0^ap_+^2(x,\lambda)e^{-2\pi in\frac xa}dx
    =-\frac a{4\pi^2n^2}\int_0^a(p_+^2(x,\lambda))''e^{-2\pi in\frac xa}dx.
\end{equation*}
Thus
\begin{equation}\label{eq estimate of b-n}
    |b_n(\lambda)|\le\frac a{4\pi^2n^2}\int_0^a|(p_+^2(x,\lambda))''|dx.
\end{equation}
In terms of the corresponding Bloch solution the second derivative of $p_+^2$ is
\begin{multline*}
    (p_+^2(x,\lambda))''=2e^{-\frac{2ik(\lambda)x}a}
    \l(\psi_+''(x,\lambda)\psi_+(x,\lambda)-\frac{2k^2(\lambda)}{a^2}\psi_+^2(x,\lambda)
    \r.
    \\
    \l.
    +(\psi_+'(x,\lambda))^2
    -\frac{4ik(\lambda)}a\psi_+'(x,\lambda)\psi_+(x,\lambda)
    \r).
\end{multline*}
Let us estimate the norm in $L_1(0;a)$ of the function $\psi_+''(\cdot,\lambda)$.
From the equation
\begin{equation*}
    \psi_+''(x,\lambda)=(q(x)-\lambda)\psi_+(x,\lambda),
\end{equation*}
we see that
\begin{equation*}
    \|\psi_+''(\cdot,\lambda)\|_{L_1(0;a)}\le(\|q\|_{L_1(0;a)}+|\lambda|a)
    \max_{x\in[0;a]}|\psi_+(x,\lambda)|.
\end{equation*}
Since the functions
\begin{equation*}
    e^{-\frac{ik(\lambda)x}a},\psi_+(x,\lambda),\psi_+'(x,\lambda)
\end{equation*}
are continuous in both variables on the set
\begin{equation*}
    [0;a]\times U(\beta,\mu)
\end{equation*}
and hence attain their maximums, we have: the integral
\begin{equation*}
    \int_0^a|(p_+^2(x,\lambda))''|dx
\end{equation*}
is bounded on $U(\beta,\mu)$. For n=0,
\begin{equation*}
    b_0(\lambda)=\frac1a\int_0^ap_+^2(x,\lambda)dx
    =\frac1a\int_0^ae^{-2ik(\lambda)\frac xa}\psi_+^2(x,\lambda)dx
\end{equation*}
and is also bounded. The same argument is valid for $\hat b_n$.
Finally we see that there exists $c_5(\beta,\mu)$ such that for
every $\lambda\in U(\beta,\mu)$
\begin{equation*}
    |b_n(\lambda)|,|\hat
    b_n(\lambda)|\le\frac{c_5(\beta,\mu)}{n^2+1}.
\end{equation*}

The Wronskian $W(\psi_+(\lambda),\psi_-(\lambda))$ does not have
zeros in $U(\beta,\mu)$. Also in the formula for the derivative of
$Q$,
\begin{multline*}
    Q'(x,\lambda,\mu)=
    \frac{c\sin(2\omega x+\delta)}{(x+1)^{\gamma}W(\psi_+(\lambda),\psi_-(\lambda))}
    \left(%
    \begin{array}{cc}
    0 & -p_+^2(x,\lambda) \\
    p_-^2(x,\lambda) & 0 \\
    \end{array}%
    \right)
    \\
    -
    \l[
    Q(x,\lambda,\mu),
    \left(%
    \begin{array}{cc}
    -\frac{ik(\lambda)}a & 0 \\
    0 & \frac{ik(\lambda)}a \\
    \end{array}%
    \right)
    \r],
\end{multline*}
the functions $\pm\frac{k(\lambda)}a,\pm c\sin(2\omega
x+\delta)p_{\pm}^2(x,\lambda)$ are bounded for
$(x;\lambda)\in[0;+\infty)\times U(\beta,\mu)$. Hence there exists
$c_1(\beta,\mu,\gamma)$ such that for every $\lambda\in
U(\beta,\mu)$ and $x\ge0$ the following estimates hold
\begin{equation*}
    \|Q(x,\lambda,\mu)\|,\|Q'(x,\lambda,\mu)\|\le\frac{c_1(\beta,\mu,\gamma)}{(x+1)^{\gamma}}.
\end{equation*}
This completes the proof of the theorem.
\end{proof}

Let us study the properties of the remainder
\begin{multline}\label{eq R-(2)}
    R^{(2)}(x,\lambda,\mu):=e^{-Q(x,\lambda,\mu)}
    \l[
    \left(%
    \begin{array}{cc}
    -\frac{ik(\lambda)}a & 0 \\
    0 & \frac{ik(\lambda)}a \\
    \end{array}%
    \right)
    \r.
    \\
    +\frac{c\sin(2\omega x+\delta)}{(x+1)^{\gamma}W(\psi_+(\lambda),\psi_-(\lambda))}
    \left(%
    \begin{array}{cc}
    -p_+(x,\lambda)p_-(x,\lambda) & -p_+^2(x,\lambda) \\
    p_-^2(x,\lambda) & p_+(x,\lambda)p_-(x,\lambda) \\
    \end{array}%
    \right)
    \\
    \l.
    +R^{(1)}(x,\lambda)\r]e^{Q(x,\lambda,\mu)}
    -
    \left(%
    \begin{array}{cc}
    -\frac{ik(\lambda)}a & 0 \\
    0 & \frac{ik(\lambda)}a \\
    \end{array}%
    \right)
    \\
    -\frac{c\sin(2\omega x+\delta)}{(x+1)^{\gamma}W(\psi_+(\lambda),\psi_-(\lambda))}
    \left(%
    \begin{array}{cc}
    -p_+(x,\lambda)p_-(x,\lambda) & 0 \\
    0 & p_+(x,\lambda)p_-(x,\lambda) \\
    \end{array}%
    \right)
    \\
    -e^{-Q(x,\lambda,\mu)}\l(e^{Q(x,\lambda,\mu)}\r)'.
\end{multline}

    \begin{lem}\label{lem remainder}
    The remainder $R^{(2)}$ given by \eqref{eq R-(2)} possesses
    the following properties:
    \begin{enumerate}
    \item $R^{(2)}\in L_1(0;\infty)$ and the integral
    \begin{equation*}
        \int_0^{\infty}\|R^{(2)}(x,\lambda,\mu)\|dx
    \end{equation*}
    converges uniformly with respect to $\lambda\in U(\beta,\mu)$.
    \item
    \begin{equation}\label{eq conjugation property of R-(2)}
        \begin{array}{c}
        (R^{(2)}(x,\mu,\mu))_{21}=\overline{(R^{(2)}(x,\mu,\mu))_{12}},
        \\
        (R^{(2)}(x,\mu,\mu))_{22}=\overline{(R^{(2)}(x,\mu,\mu))_{11}}.
        \end{array}
    \end{equation}
    \end{enumerate}
    \end{lem}

\begin{proof}
The first assertion follows directly from Theorem \ref{thm about
Q}.

The property of matrices from the second assertion is preserved under summation and
multiplication of such matrices as well as taking the inverse. We can see from
\eqref{eq R-(1)} and \eqref{eq formula for Q on the real line} that $R^{(1)}(x,\mu)$ and
$Q(x,\mu,\mu)$ possess this conjugation property. Therefore
\begin{equation*}
    Q'(x,\mu,\mu),e^{Q(x,\mu,\mu)},\l(e^{Q(x,\mu,\mu)}\r)'
\end{equation*}
and finally $R^{(2)}(x,\mu,\mu)$ (from \eqref{eq R-(2)}) also
possess this property. It should be taken into account that
$W(\psi_+(\mu),\psi_-(\mu))$ is pure imaginary.
\end{proof}

Let us denote
\begin{equation*}
    \nu(x,\lambda):=-\frac{ik(\lambda)}a-\frac{c\sin(2\omega x+\delta)p_+(x,\lambda)p_-(x,\lambda)}
    {(x+1)^{\gamma}W(\psi_+(\lambda),\psi_-(\lambda))}.
\end{equation*}
Finally we obtain a system of the Levinson form:
\begin{equation}\label{eq system for widetilde-v}
    \widetilde v'(x)=
    \l[
    \left(%
    \begin{array}{cc}
    \nu(x,\lambda) & 0 \\
    0 & -\nu(x,\lambda) \\
    \end{array}%
    \right)
    +R^{(2)}(x,\lambda,\mu)
    \r]
    \widetilde v(x).
\end{equation}

\section{A Levinson-type theorem for $2\times2$ systems}
In this section, we prove two statements that give a uniform
estimate and asymptotics of solutions to certain $ 2 \times 2 $
differential systems. The approach is the same as for the Levinson
theorem \cite{Coddington-Levinson-55}, but the difference is that
we are interested in properties of solution with a given initial
condition.

Consider the system
\begin{equation}\label{eq Levinson system}
    u_1'(x)=
    \l[
    \left(%
    \begin{array}{cc}
    \lambda(x) & 0 \\
    0 & -\lambda(x) \\
    \end{array}%
    \right)
    +R(x)
    \r]
    u_1(x)
\end{equation}
for $x\ge0$, where $u_1(x) $ is a two-dimensional vector function
and $ R(x) $ is a $2\times2$ matrix with complex entries.

    \begin{lem}\label{lem Levinson estimate}
    Assume that
    \begin{equation}\label{eq condition of summability}
        \int_0^{\infty}\|R(t)\|dt<\infty.
    \end{equation}
    and that there exists a constant $M$ such that for every $x\le
    y$ it holds
    \begin{equation*}
        \int_x^y \Re  \lambda(t)dt\ge-M.
    \end{equation*}
    Then every solution $u_1$ to \eqref{eq Levinson system} satisfies
    the estimate
    \begin{equation}\label{eq estimate on growth of solution}
        \|u_1(x)\|\le\|u_1(0)\|e^{\int_0^x\Re \lambda(t)dt}\sqrt{1+e^{4M}}
        e^{\sqrt{1+e^{4M}}\int_0^{\infty}\|R(t)\|dt}.
    \end{equation}
    \end{lem}

\begin{proof}
First transform the system \eqref{eq Levinson system} by variation
of parameters. Denote
\begin{equation*}
    \Lambda(x):=
    \left(%
    \begin{array}{cc}
    \lambda(x) & 0 \\
    0 & -\lambda(x) \\
    \end{array}%
    \right)
\end{equation*}
and take
\begin{equation}\label{eq u-2}
    u_1(x)=e^{\int_0^x\Lambda(t)dt}u_2(x),\text{ or }
    u_2(x):=e^{-\int_0^x\Lambda(t)dt}u_1(x).
\end{equation}
After substitution \eqref{eq Levinson system} becomes
\begin{equation}\label{eq u-2'}
    u_2'(x)=e^{-\int_0^x\Lambda(t)dt}R(x)u_1(x).
\end{equation}
Integrating this from $0$ to $x$ and returning back to the
function $u_1$ on the left-hand side we get
\begin{equation}\label{eq integral equation for u-1}
    u_1(x)=e^{\int_0^x\Lambda(t)dt}u_1(0)+\int_0^xe^{\int_t^x\Lambda(s)ds}R(t)u_1(t)dt.
\end{equation}
Now multiply this expression by $e^{-\int_0^x\Lambda(s)ds}$ and
denote
\begin{equation}\label{eq u-3}
    u_3(x):=e^{-\int_0^x\Lambda(t)dt}u_1(x).
\end{equation}
We get the following equation for $ u_3 $ considered in $ L_\infty
(0,\infty); \mathbb C^2)$
\begin{equation}\label{eq integral equation for u-3}
    u_3(x)=
    \left(%
    \begin{array}{cc}
    1 & 0 \\
    0 & e^{-2\int_0^x\lambda(s)ds} \\
    \end{array}%
    \right) u_1(0)+\int_0^x
    \left(%
    \begin{array}{cc}
    1 & 0 \\
    0 & e^{-2\int_t^x\lambda(s)ds} \\
    \end{array}%
    \right) R(t)u_3(t)dt.
\end{equation}
 The norm of the operator $V$,
\begin{equation*}
    V:u\mapsto\int_0^x
    \left(%
    \begin{array}{cc}
    1 & 0 \\
    0 & e^{-2\int_t^x\lambda(s)ds} \\
    \end{array}%
    \right) R(t)u(t)dt
\end{equation*}
is bounded by
\begin{equation*}
    \|V\|\le\sqrt{1+e^{4M}}\int_0^{\infty}\|R(t)\|dt
\end{equation*}
and the norm of the $j$-th power is bounded by
\begin{equation*}
    \|V^j\|\le\frac{(\sqrt{1+e^{4M}}\int_0^{\infty}\|R(t)\|dt)^j}{j!}.
\end{equation*}
Hence
\begin{equation*}
    u_3(x)=(I-V)^{-1}
       \left(%
    \begin{array}{cc}
    1 & 0 \\
    0 & e^{-2\int_0^x\lambda(s)ds} \\
    \end{array}%
    \right)
    u_1(0)
\end{equation*}
and
\begin{equation*}
    \|u_3\|_{L_{\infty}((0;\infty),\mathbb C^2)}
    \le\exp\l(\sqrt{1+e^{4M}}\int_0^{\infty}\|R(t)\|dt\r)\sqrt{1+e^{4M}}
    \|u_1(0)\|.
\end{equation*}
Returning to $u_1$, we arrive at the estimate \eqref{eq estimate on growth of solution}.
\end{proof}

The second lemma states the asymptotics of the solution.

    \begin{lem}\label{lem Levinson asymptotics}
    Let that all conditions of Lemma \ref{lem Levinson    estimate}
    be satisfied, then the following asymptotics hold:
  \begin{enumerate}
    \item If
    \begin{equation}\label{eq condition Levinson elliptic}
    \int_0^{\infty}\Re \lambda(t)dt<+\infty,
    \end{equation}
    then every solution $u_1$ of \eqref{eq Levinson system} has the following asymptotics:
    \begin{multline*}
    u_1(x)=
    \left(%
    \begin{array}{cc}
    e^{\int_0^x\lambda(s)ds} & 0 \\
    0 & e^{-\int_0^x\lambda(s)ds} \\
    \end{array}%
    \right)
    \l[
    u_1(0)
    \r.
    \\
    \l.
    +\int_0^{\infty}
    \left(%
    \begin{array}{cc}
    e^{-\int_0^t\lambda(s)ds} & 0\\
    0 & e^{\int_0^t\lambda(s)ds} \\
    \end{array}%
    \right)
    R(t)u_1(t)dt+o(1)
    \r]\as x\rightarrow\infty.
    \end{multline*}
    \item If
    \begin{equation}\label{eq condition Levinson hyperbolic}
    \int_0^{\infty}\Re \lambda(t)dt=+\infty,
    \end{equation}
    then every solution $u_1$ of \eqref{eq Levinson system} has the following asymptotics:
    \begin{multline*}
    u_1(x)=e^{\int_0^x\lambda(s)ds}
    \l[
    \left(%
    \begin{array}{cc}
    1 & 0 \\
    0 & 0 \\
    \end{array}%
    \right)
    \l(
    u_1(0)
    \r.
    \r.
    \\
    \l.
    \l.
    +\int_0^{\infty}
    e^{-\int_0^t\lambda(s)ds}R(t)u_1(t)dt\r)+o(1)
    \r]\as x\rightarrow\infty.
    \end{multline*}
    \end{enumerate}
    \end{lem}

\begin{proof}
Asymptotics 1. Consider the function $u_2$ given by \eqref{eq u-2}
and integrate \eqref{eq u-2'}:
\begin{equation}\label{eq integral equation for u-2}
    u_2(x)=u_1(0)+\int_0^xe^{-\int_0^t\Lambda(s)ds}R(t)u_1(t)dt.
\end{equation}
Since for every $x\le y$ we have the estimate
\begin{multline*}
    \int_x^y\Re \lambda(s)ds\le\int_x^y|\Re \lambda(s)|ds
    \\
    \le\int_0^{\infty}|\Re \lambda(s)|ds\le\int_0^{\infty}\Re \lambda(s)ds+2M,
\end{multline*}
the exponent under the integral in \eqref{eq integral equation for
u-2} is bounded. The solution $u_1(t)$ is also bounded in this
case due to Lemma \ref{lem Levinson estimate}. Hence the integral
in \eqref{eq integral equation for u-2} converges as
$x\rightarrow\infty$ and there exists
\begin{equation*}
    \lim_{x\rightarrow\infty}u_2(x)=u_1(0)+\int_0^{\infty}e^{-\int_0^t\Lambda(s)ds}R(t)u_1(t)dt.
\end{equation*}
Returning to $u_1$ we obtain the answer.

Asymptotics 2. Consider the function $u_3$ given by \eqref{eq u-3}
and the corresponding equation \eqref{eq integral equation for
u-3}. It follows from Lemma \ref{lem Levinson estimate}  that
$u_3(t)$ is bounded, and hence Lebesgue's dominated convergence
theorem implies that the following limit exists
\begin{equation*}
    \lim_{x\rightarrow\infty}u_3(x)=
    \left(%
    \begin{array}{cc}
    1 & 0 \\
    0 & 0 \\
    \end{array}%
    \right)
    \l[u_1(0)+\int_0^{\infty}R(t)u_3(t)dt\r].
\end{equation*}
This is equivalent to the announced asymptotics for $u_1$.
\end{proof}

\section{Asymptotics for the solution $\varphi$ and Weyl-Titchmarsh type formula}

In this section, we obtain the asymptotics for the solution
$\varphi_{\alpha}(x,\lambda)$ and prove the Weyl-Titchmarsh type
formula for the operator $\mathcal L_{\alpha}$. Consider the set
\begin{equation*}
    U(\beta):=\bigcup_{\mu\in\sigma(\mathcal L_{per})\backslash\{\lambda_n,\mu_n,\lambda_n^+,\lambda_n^-,n\ge0\}}
    U(\beta,\mu)
\end{equation*}
that belongs to $\overline{\mathbb C_+}$ and contains
\begin{equation*}
    \sigma(\mathcal L_{per})\backslash\{\lambda_n,\mu_n,\lambda_n^+,\lambda_n^-,n\ge0\}
\end{equation*}
as a part of its boundary. The number $\beta$ is arbitrary here.

\begin{thm}\label{thm asymptotics}
    Let $\frac{2a\omega}{\pi}\notin\mathbb Z$ and $q_1\in L_1(\mathbb
    R_+)$, then the
    solution $\varphi_{\alpha}$ of the Cauchy problem
    \begin{equation*}
    \begin{array}{l}
    -\varphi_{\alpha}''(x,\lambda)
    +\l(q(x)+\frac{c\sin(2\omega x+\delta)}{(x+1)^{\gamma}}+q_1(x)\r)\varphi_{\alpha}(x,\lambda)
    =\lambda\varphi_{\alpha}(x,\lambda),
    \\
    \varphi_{\alpha}(0,\lambda)=\sin\alpha,
    \\
    \varphi'_{\alpha}(0,\lambda)=\cos\alpha
    \end{array}
    \end{equation*}
    has the following asymptotics.
     For every $\lambda\in U(\beta)$ there exists $A_{\alpha}(\lambda)$ such that
    \begin{enumerate}
    \item If $\lambda\in\mathbb C_+\cap U(\beta)$, then
    \begin{equation*}
    \begin{array}{l}
    \varphi_{\alpha}(x,\lambda)=A_{\alpha}(\lambda)\psi_-(x,\lambda)+o\l(e^{\Im  k(\lambda)\frac xa}\r),
    \\
    \varphi_{\alpha}'(x,\lambda)=A_{\alpha}(\lambda)\psi_-'(x,\lambda)+o\l(e^{\Im  k(\lambda)\frac xa}\r)
    \end{array}
    \end{equation*}
    as $x\rightarrow\infty$.
    \item If $\lambda\in\sigma(\mathcal L_{per})\backslash\{\lambda_n,\mu_n,\lambda_n^+,\lambda_n^-,n\ge0\}$,
    then
    \begin{equation*}
    \begin{array}{l}
    \varphi_{\alpha}(x,\lambda)=A_{\alpha}(\lambda)\psi_-(x,\lambda)+
    \overline{A_{\alpha}(\lambda)}\psi_+(x,\lambda)+o(1),
    \\
    \varphi_{\alpha}'(x,\lambda)=A_{\alpha}(\lambda)\psi_-'(x,\lambda)+
    \overline{A_{\alpha}(\lambda)}\psi_+'(x,\lambda)+o(1)
    \end{array}
    \end{equation*}
    as $x\rightarrow\infty$.
    \end{enumerate}

    The function $A_{\alpha}$ is analytic in the interior of $U(\beta)$ and has boundary values on
    \begin{equation*}
        \sigma(\mathcal L_{per})\backslash\{\lambda_n,\mu_n,\lambda_n^+,\lambda_n^-,n\ge0\}.
    \end{equation*}
\end{thm}

\begin{proof}
We are going to omit the index $\alpha$ since the value of the
boundary parameter is fixed throughout this proof. According to
\eqref{eq u} and \eqref{eq v} we write
\begin{equation}\label{eq v-phi}
    \left(%
    \begin{array}{c}
    \varphi(x) \\
    \varphi'(x) \\
    \end{array}%
    \right)
    =
    \left(%
    \begin{array}{cc}
    \psi_-(x,\lambda) & \psi_+(x,\lambda) \\
    \psi_-'(x,\lambda) & \psi_+'(x,\lambda) \\
    \end{array}%
    \right)
    \left(%
    \begin{array}{cc}
    e^{ik(\lambda)\frac xa} & 0 \\
    0 & e^{-ik(\lambda)\frac xa} \\
    \end{array}%
    \right)
    v_{\varphi}(x,\lambda).
\end{equation}
This is the definition of $v_{\varphi}$, a solution of \eqref{eq system for v}
corresponding to $\varphi$. Let us fix the point
\begin{equation*}
    \mu\in\sigma(\mathcal L_{per})\backslash\{\lambda_n,\mu_n,\lambda_n^+,\lambda_n^-,n\ge0\}
\end{equation*}
and consider $\lambda\in U(\beta,\mu)$. The function
\begin{equation*}
    \widetilde v_{\varphi}(x,\lambda,\mu):=e^{-Q(x,\lambda,\mu)}v_{\varphi}(x,\lambda),
\end{equation*}
is a solution to \eqref{eq system for widetilde-v} corresponding
to $\varphi$. Let us see that conditions of Lemma \ref{lem
Levinson estimate} are satisfied for the system \eqref{eq system
for widetilde-v} uniformly with respect to $\lambda\in
U(\beta,\mu)$. First of all we have estimate \eqref{eq condition
of summability} from Lemma \ref{lem remainder} and
\begin{equation*}
    \Re \nu(x,\lambda)
    =\frac{\Im  k(\lambda)}a-\Re \l(\frac{c\sin(2\omega x+\delta)p_+(x,\lambda)p_-(x,\lambda)}
    {(x+1)^{\gamma}W(\psi_+(\lambda),\psi_-(\lambda))}\r).
\end{equation*}
Estimating the second term in the same way as in Theorem \ref{thm about Q} we have:
\begin{multline*}
    \l|\int_x^y\Re \frac{c\sin(2\omega t+\delta)p_+(t,\lambda)p_-(t,\lambda)}
    {(t+1)^{\gamma}W(\psi_+(\lambda),\psi_-(\lambda))}dt\r|
    \le\frac{|c|a}{\pi|W(\psi_+(\lambda),\psi_-(\lambda))|}
    \\
    \times
    \l(\sum_{n=-\infty}^{\infty}|\widetilde b_n(\lambda)|
    \l(\frac1{\l|\frac{2a\omega}{\pi}+2n\r|}+\frac1{\l|\frac{2a\omega}{\pi}-2n\r|}\r)\r)
    \frac1{(x+1)^{\gamma}},
\end{multline*}
where
\begin{equation*}
    \widetilde b_n(\lambda):=\frac1a\int_0^ap_+(x,\lambda)p_-(x,\lambda)e^{-2\pi in\frac xa}dx
\end{equation*}
are Fourier coefficients for
$p_+(\cdot,\lambda)p_-(\cdot,\lambda)$. Analogously to \eqref{eq
estimate of b-n} we have:
\begin{equation*}
    |\widetilde b_n(\lambda)|\le\frac a{4\pi^2n^2}\int_0^a|(\psi_+(x,\lambda)\psi_-(x,\lambda))''|dx.
\end{equation*}
So there exists $c_6(\beta,\mu)$ such that for every $\lambda\in U(\beta,\mu)$ and $n\neq0$
\begin{equation*}
    |\widetilde b_n(\lambda)|\le\frac{c_6(\beta,\mu)}{n^2},
\end{equation*}
while
\begin{equation*}
    |\widetilde b_0(\lambda)|\le c_6(\beta,\mu).
\end{equation*}
Eventually there exists $c_7(\beta,\mu)$ such that
\begin{equation*}
    \l|\int_x^y\Re \frac{c\sin(2\omega t+\delta)p_+(t,\lambda)p_-(t,\lambda)}
    {(t+1)^{\gamma}W(\psi_+(\lambda),\psi_-(\lambda))}dt\r|
    \le c_7(\beta,\mu)
\end{equation*}
for every $0\le x\le y$ and $\lambda\in U(\beta,\mu)$. Thus we can take
\begin{equation*}
    M(\lambda)\equiv c_7(\beta,\mu)
\end{equation*}
for these values of $\lambda$. Lemma \ref{lem Levinson estimate} gives the estimate
\begin{equation}\label{eq estimate of widetilde-v-phi}
    \|\widetilde v_{\varphi}(x,\lambda,\mu)\|\le\|\widetilde v_{\varphi}(0,\lambda,\mu)\|
    e^{\Im  k(\lambda)\frac xa}c_8(\beta,\mu),
\end{equation}
where
\begin{multline*}
    c_8(\beta,\mu):=\sqrt{1+e^{4c_7(\beta,\mu)}}
    \\
    \times
    \exp\l(\sqrt{1+e^{4c_7(\beta,\mu)}}\max\limits_{\lambda\in U(\beta,\mu)}
    \int_0^{\infty}\|R^{(2)}(t,\lambda)\|dt\r).
\end{multline*}
Conditions of Lemma \ref{lem Levinson asymptotics} are also
satisfied: \eqref{eq condition Levinson elliptic} holds for
$\lambda\in\mathbb R\cap U(\beta,\mu)$ and \eqref{eq condition
Levinson hyperbolic} holds for $\lambda\in\mathbb C_+\cap
U(\beta,\mu)$. So Lemma \ref{lem Levinson asymptotics} gives the
following asymptotics: \begin{itemize} \item for
$\lambda\in\mathbb C_+\cap U(\beta,\mu)$,
\begin{multline*}
    \widetilde v_{\varphi}(x,\lambda,\mu)
    =e^{-ik(\lambda)\frac xa-\int_0^x\frac{c\sin(2\omega t+\delta)p_+(t,\lambda)p_-(t,\lambda)dt}
    {(t+1)^{\gamma}W(\psi_+(\lambda),\psi_-(\lambda))}}
    \l[
    \left(%
    \begin{array}{cc}
    1 & 0 \\
    0 & 0 \\
    \end{array}%
    \right)
    \l(
    \widetilde v_{\varphi}(0,\lambda,\mu)
    \r.
    \r.
    \\
    \l.
    \l.
    +\int_0^{\infty}
    e^{ik(\lambda)\frac ta+\int_0^t\frac{c\sin(2\omega s+\delta)p_+(s,\lambda)p_-(s,\lambda)ds}
    {(s+1)^{\gamma}W(\psi_+(\lambda),\psi_-(\lambda))}}
    R^{(2)}(t,\lambda,\mu)\widetilde v_{\varphi}(t,\lambda,\mu)dt
    \r)
    +o(1)\r],
\end{multline*}
\item for $\lambda=\mu$,
\begin{multline}\label{eq asymptotics of widetilde-v-phi}
    \widetilde v_{\varphi}(x,\mu,\mu)
    =
    \left(%
    \begin{array}{l}
    e^{-ik(\mu)\frac xa-\int_0^x\frac{c\sin(2\omega t+\delta)p_+(t,\mu)p_-(t,\mu)dt}
    {(t+1)^{\gamma}W(\psi_+(\mu),\psi_-(\mu))}}
    \ \ \ \
    0
    \\
    \ \ \
    0
    \ \ \ \ \
    e^{ik(\mu)\frac xa+\int_0^x\frac{c\sin(2\omega t+\delta)p_+(t,\mu)p_-(t,\mu)dt}
    {(t+1)^{\gamma}W(\psi_+(\mu),\psi_-(\mu))}}
    \\
    \end{array}%
    \right)
    \\
    \times
    \l[
    \widetilde v_{\varphi}(0,\mu,\mu)
    +
    \int_0^{\infty}
    \left(%
    \begin{array}{l}
    e^{ik(\mu)\frac ta+\int_0^t\frac{c\sin(2\omega s+\delta)p_+(s,\mu)p_-(s,\mu)ds}
    {(s+1)^{\gamma}W(\psi_+(\mu),\psi_-(\mu))}}
    \ \ \ \ \ \
    0
    \\
    \ \ \
    0
    \ \ \ \ \
    e^{-ik(\mu)\frac ta-\int_0^t\frac{c\sin(2\omega s+\delta)p_+(s,\mu)p_-(s,\mu)ds}
    {(s+1)^{\gamma}W(\psi_+(\mu),\psi_-(\mu))}}
    \\
    \end{array}%
    \right)
    \r.
    \\
    \l.
    \times
    R^{(2)}(t,\mu,\mu)\widetilde v_{\varphi}(t,\mu,\mu)dt+o(1)\r].
\end{multline}
\end{itemize}

Since $Q(x,\lambda,\mu)=O\l(\frac1{(x+1)^{\gamma}}\r)$, we can denote for $\lambda\in U(\beta,\mu)$:
\begin{multline*}
    A(\lambda,\mu):=
    \l<
    \left(%
    \begin{array}{c}
    1 \\
    0 \\
    \end{array}%
    \right),
    e^{-\int_0^{\infty}\frac{c\sin(2\omega t+\delta)p_+(t,\lambda)p_-(t,\lambda)dt}{(t+1)^{\gamma}W(\psi_+(\lambda),\psi_-(\lambda))}}
    \l[
    e^{-Q(0,\lambda,\mu)}v_{\varphi}(0,\lambda)
    \r.
    \r.
    \\
    \l.
    \l.
    +\int_0^{\infty}
    e^{ik(\lambda)\frac ta+\int_0^t\frac{c\sin(2\omega s+\delta)p_+(s,\lambda)p_-(s,\lambda)ds}
    {(s+1)^{\gamma}W(\psi_+(\lambda),\psi_-(\lambda))}}
    R^{(2)}(t,\lambda,\mu)e^{Q(t,\lambda,\mu)}v_{\varphi}(t,\lambda)dt\r]
    \r>
\end{multline*}
(where $<\cdot,\cdot>$ stands for the scalar product in $\mathbb C^2$) and have:
\begin{equation}\label{eq limit in hyperbolic case for v-phi}
    \lim_{x\rightarrow\infty}v_{\varphi}(x,\lambda)e^{ik(\lambda)\frac xa}
    =
    \left(%
    \begin{array}{c}
    A(\lambda,\mu) \\
    0 \\
    \end{array}%
    \right).
\end{equation}
From this we see that the coefficient $A(\lambda,\mu)$ does not depend on $\mu$,
so we will denote it by $A(\lambda)$. Relation \eqref{eq v-phi} can be written as
\begin{equation}\label{eq v-phi 2 time}
    v_{\varphi}(x,\lambda)=\frac1{W(\psi_+(\lambda),\psi_-(\lambda))}
    \left(%
    \begin{array}{c}
    \psi_+'(x,\lambda)\varphi(x,\lambda)-\psi_+(x,\lambda)\varphi'(x,\lambda) \\
    \varphi'(x,\lambda)\psi_-(x,\lambda)-\varphi(x,\lambda)\psi_-'(x,\lambda) \\
    \end{array}%
    \right),
\end{equation}
so $v_{\varphi}(x,\cdot)$ is analytic in $\mathbb C_+$ and continuous up to
\begin{equation*}
    \sigma(\mathcal L_{per})\backslash\{\lambda_n,\mu_n,n\ge0\}.
\end{equation*}
From the estimate \eqref{eq estimate of widetilde-v-phi} and properties of $Q(x,\lambda,\mu)$
and $R^{(2)}(x,\lambda,\mu)$ given by Theorem \ref{thm about Q} and Lemma \ref{lem remainder}
it follows that $A(\lambda)$ is continuous in $U(\beta,\mu)$ and analytic in its interior.
Thus $A$ is analytic in the interior of $U(\beta)$ having non-tangential boundary limits
on
\begin{equation*}
    \sigma(\mathcal L_{per})\backslash\{\lambda_n,\mu_n,\lambda_n^+,\lambda_n^-,n\ge0\}.
\end{equation*}
that coincide with its values on this set.

The solution $\varphi(x,\lambda)$ and its derivative are real if
$\lambda$ is real. Thus \eqref{eq v-phi 2 time} shows that the
upper and the lower components of the vector
$v_{\varphi}(x,\lambda)$ are complex conjugate for
$\lambda\in\sigma(\mathcal
L_{per})\backslash\{\lambda_n,\mu_n,n\ge0\}$. This property is
preserved if we multiply the vector by a matrix $X$ such that
\begin{equation*}
    X_{21}=\overline{X_{12}},\ X_{22}=\overline{X_{11}},
\end{equation*}
like \eqref{eq conjugation property of R-(2)}. It follows from
Lemma \ref{lem remainder} that the upper and the lower components
of the vectors in the equality \eqref{eq asymptotics of
widetilde-v-phi} are complex conjugate to each other. Hence for
$\lambda=\mu$ we have
\begin{equation}\label{eq asymptotics of v-phi in elliptic case}
    v_{\varphi}(x,\mu)=
    \left(%
    \begin{array}{c}
    A(\mu)e^{-ik(\mu)\frac xa} \\
    \overline{A(\mu)}e^{ik(\mu)\frac xa} \\
    \end{array}%
    \right)
    +o(1)\as x\rightarrow\infty.
\end{equation}
The asymptotics of the solution $\varphi$ and its derivative
follows from \eqref{eq v-phi}, \eqref{eq limit in hyperbolic case
for v-phi} and \eqref{eq asymptotics of v-phi in elliptic case}.
\end{proof}

Using the obtained asymptotics both on the spectrum and in
$\mathbb C_+$ we now prove the Weyl-Titchmarsh type formula.

    \begin{thm}\label{thm Weyl-Titchmarsh formula}
    Let $\frac{2a\omega}{\pi}\notin\mathbb Z$ and $q_1\in L_1(\mathbb R_+)$,
    then for almost all $\lambda\in\sigma(\mathcal
    L_{per})$ the spectral density of the operator $\mathcal
    L_{\alpha}$,
    defined
    by \eqref{eq L-alpha}, is given by
    \begin{equation*}
        \rho'_{\alpha}(\lambda)=\frac1{2\pi|W(\psi_+(\lambda),\psi_-(\lambda))||A_{\alpha}(\lambda)|^2},
    \end{equation*}
    where $A_{\alpha}$ is the same as in Theorem \ref{thm asymptotics}.
    \end{thm}

\begin{proof}
In addition to $\varphi_{\alpha}$ consider another one solution of
\eqref{eq spectral equation 2 time}, to be denoted by
$\theta_{\alpha}:=\varphi_{\alpha+\frac{\pi}2}$, satisfying the
initial conditions
\begin{equation*}
    \theta_{\alpha}(0,\lambda)=\cos\alpha,\ \theta'_{\alpha}(0,\lambda)=-\sin\alpha.
\end{equation*}
The Wronskian of $\varphi_{\alpha}$ and $\theta_{\alpha}$ is equal
to one. Theorem \ref{thm asymptotics} yields for $\lambda\in
U(\beta)\cap\mathbb C_+$,
\begin{equation*}
    \theta_{\alpha}(x,\lambda)=A_{\alpha+\frac{\pi}2}(\lambda)\psi_-(x,\lambda)
    +o\l(e^{\Im  k(\lambda)\frac xa}\r)\as x\rightarrow\infty.
\end{equation*}
Since the operator $\mathcal L_{\alpha}$ is in the limit point case,
the combination
\begin{equation*}
    \theta_{\alpha}+m_{\alpha}\varphi_{\alpha}
\end{equation*}
belongs to $L_2(0,\infty)$ (where $m_{\alpha}$ is the Weyl
function for $\mathcal L_{\alpha}$). It has the asymptotics
\begin{equation*}
    \theta_{\alpha}(x,\lambda)+m_{\alpha}(\lambda)\varphi_{\alpha}(x,\lambda)
    =(A_{\alpha+\frac{\pi}2}(\lambda)+m_{\alpha}A_{\alpha}(\lambda))
    \psi_-(x,\lambda)+o(e^{\Im  k(\lambda)\frac xa}).
\end{equation*}
Therefore
\begin{equation*}
    m_{\alpha}(\lambda)=-\frac{A_{\alpha+\frac{\pi}2}(\lambda)}{A_{\alpha}(\lambda)}
\end{equation*}
for $\lambda\in U(\beta)\cap\mathbb C_+$ and
\begin{equation*}
    m_{\alpha}(\lambda+i0)=-\frac{A_{\alpha+\frac{\pi}2}(\lambda)}{A_{\alpha}(\lambda)}
\end{equation*}
for $\lambda\in\sigma(\mathcal
L_{per})\backslash\{\lambda_n,\mu_n,\lambda_n^+,\lambda_n^-,n\ge0\}$.
It follows from the subordinacy theory \cite{Gilbert-Pearson-87}
that the spectrum of $\mathcal L_{\alpha}$ on this set is purely
absolutely continuous and
\begin{equation}\label{eq rho}
    \rho'_{\alpha}(\lambda)=\frac1{\pi}\Im  m_{\alpha}(\lambda+i0)=
    \frac{A_{\alpha}(\lambda)\overline{A_{\alpha+\frac{\pi}2}(\lambda)}
    -\overline{A_{\alpha}(\lambda)}A_{\alpha+\frac{\pi}2}(\lambda)}
    {2\pi i|A_{\alpha}(\lambda)|^2}.
\end{equation}
Theorem \ref{thm asymptotics} yields for these values of $\lambda$:
\begin{equation*}
    \begin{array}{l}
    \theta_{\alpha}(x,\lambda)=A_{\alpha+\frac{\pi}2}(\lambda)\psi_-(x,\lambda)
    +\overline{A_{\alpha+\frac{\pi}2}(\lambda)}\psi_+(x,\lambda)+o(1),
    \\
    \theta'_{\alpha}(x,\lambda)=A_{\alpha+\frac{\pi}2}(\lambda)\psi'_-(x,\lambda)
    +\overline{A_{\alpha+\frac{\pi}2}(\lambda)}\psi'_+(x,\lambda)+o(1),
    \end{array}
\end{equation*}
as $x\rightarrow\infty$. Substituting these asymptotics and the
asymptotics of $\varphi_{\alpha}$ and $\varphi'_{\alpha}$ into the
expression for the Wronskian we get:
\begin{equation*}
    1=(\overline{A_{\alpha}(\lambda)}A_{\alpha+\frac{\pi}2}(\lambda)
    -A_{\alpha}(\lambda)\overline{A_{\alpha+\frac{\pi}2}(\lambda)})
    W(\psi_+(\lambda),\psi_-(\lambda))
\end{equation*}
(the term $o(1)$ cancels, since both sides are independent of 
$x$). Combining with \eqref{eq rho} we have
\begin{multline*}
    \rho'_{\alpha}(\lambda)
    =\frac1{-2\pi iW(\psi_+(\lambda),\psi_-(\lambda))|A_{\alpha}(\lambda)|^2}
    \\
    =\frac1{2\pi|W(\psi_+(\lambda),\psi_-(\lambda))||A_{\alpha}(\lambda)|^2},
\end{multline*}
which completes the proof.
\end{proof}

\section{Acknowledgements}
The second author expresses his deep gratitude to Professor
S.N.\,Naboko for his constant attention to this work and for many
fruitful discussions on the subject and also to the Mathematics
Department of Lund Institute of Technology for financial support
and hospitality. The work was supported by grants
RFBR-09-01-00515-a, INTAS-05-1000008-7883 and Swedish Research
Council 80525401.

\bibliographystyle{plain}

\end{document}